\documentclass[12pt,a4paper]{amsart}
\usepackage{amssymb, amscd}

\theoremstyle{plain}
\newtheorem{theorem}{Theorem}[subsection]
\newtheorem{lemma}[theorem]{Lemma}
\newtheorem{proposition}[theorem]{Proposition}
\newtheorem{corollary}[theorem]{Corollary}

\theoremstyle{definition}

\newtheorem{examples}[theorem]{Examples}

\numberwithin{equation}{subsection}

\newcommand{\diag}{\operatorname{diag}}
\newcommand{\divi}{\operatorname{div}}
\newcommand{\id}{\operatorname{id}}
\newcommand{\modu}{\operatorname{mod}}
\newcommand{\rk}{\operatorname{rk}}
\newcommand{\td}{\operatorname{td}}

\newcommand{\GL}{\operatorname{GL}}
\newcommand{\Hom}{\operatorname{Hom}}
\newcommand{\Pic}{\operatorname{Pic}}
\newcommand{\Proj}{\operatorname{Proj}}
\newcommand{\PGL}{\operatorname{PGL}}
\newcommand{\PSL}{\operatorname{PSL}}
\newcommand{\PSO}{\operatorname{PSO}}
\newcommand{\PSp}{\operatorname{PSp}}
\newcommand{\SL}{\operatorname{SL}}
\newcommand{\SO}{\operatorname{SO}}
\newcommand{\Spec}{\operatorname{Spec}}

\newcommand{\bbA}{\mathbb A}

\newcommand{\bbG}{\mathbb G}
\newcommand{\bbP}{\mathbb P}
\newcommand{\bbQ}{\mathbb Q}
\newcommand{\bbR}{\mathbb R}
\newcommand{\bbZ}{\mathbb Z}

\newcommand{\bG}{\mathbf G}
\newcommand{\bT}{\mathbf T}

\newcommand{\cE}{\mathcal E}
\newcommand{\cL}{\mathcal L}
\newcommand{\cN}{\mathcal N}
\newcommand{\cO}{\mathcal O}
\newcommand{\cS}{\mathcal S}
\newcommand{\cT}{\mathcal T}
\newcommand{\cX}{\mathcal X}

\newcommand{\fa}{\mathfrak a}
\newcommand{\fg}{\mathfrak g}

\newcommand{\fk}{\mathfrak k}
\newcommand{\fl}{\mathfrak l}
\newcommand{\fp}{\mathfrak p}
\newcommand{\ft}{\mathfrak t}

\author{M.~Brion and R.~Joshua}
\address{Michel Brion\\
Universit\'e de Grenoble--I\\
Institut Fourier, BP 74\\ 
38402 Saint-Martin d'H\`eres Cedex\\
France}
\email{Michel.Brion@ujf-grenoble.fr}

\address{Roy Joshua\\ 
Department of Mathematics\\  
Ohio State University\\
231 W 18th Avenue\\
Columbus, OH 43210\\
USA}
\email{joshua@math.ohio-state.edu}

\begin{document}

\begin{abstract}
We describe the equivariant Chow ring of the wonderful
compactification $X$ of a symmetric space of minimal rank, via
restriction to the associated toric variety $Y$. Also, we show that
the restrictions to $Y$ of the tangent bundle $T_X$ and its
logarithmic analogue $S_X$ decompose into a direct sum of line
bundles. This yields closed formulae for the equivariant Chern classes
of $T_X$ and $S_X$, and, in turn, for the Chern classes of reductive
groups considered by Kiritchenko.
\end{abstract}

\title[Equivariant Chow ring of wonderful varieties of minimal rank]
{Equivariant Chow ring and Chern classes of wonderful symmetric
varieties of minimal rank} 

\date{}

\thanks{The second author thanks the IHES, the MPI and the NSA for support}

\maketitle

\setcounter{section}{-1}

\section{Introduction}
\label{sec:introduction}

The purpose of this article is to describe the equivariant
intersection ring and equivariant Chern classes of a class of almost
homogeneous varieties, namely, wonderful symmetric varieties of
minimal rank; these include the wonderful compactifications of 
semi-simple groups of adjoint type.  

\smallskip

The main motivation comes from questions of enumerative geometry on a
spherical homogeneous space $G/K$. As shown by De Concini and Procesi,
these questions find their proper setting in the ring of conditions
$C^*(G/K)$, isomorphic to the direct limit of cohomology rings of
$G$-equivariant compactifications $X$ of $G/K$ (see \cite{DP83,DP85}). 
Recently, the Euler characteristic of any complete intersection of
hyper-surfaces in $G/K$ has been expressed by Kiritchenko (see
\cite{Ki06}), in terms of the Chern classes of the logarithmic tangent
bundle $S_X$ of any ``regular'' compactification $X$. As shown in
\cite{Ki06}, these Chern classes are independent of the choice of 
$X$, and hence yield elements of $C^*(G/K)$; moreover, their
determination may be reduced to the case where $X$ is a 
``wonderful variety''. 

\smallskip

In fact, it is more convenient to work with the rational equivariant
cohomology ring $H^*_G(X)$, from which the ordinary rational
cohomology ring $H^*(X)$ is obtained by killing the action of
generators of the polynomial ring $H^*(BG)$; the Chern classes of
$S_X$ have natural representatives in $H^*_G(X)$, the equivariant
Chern classes. When $X$ is a complete symmetric variety, the ring
$H^*_G(X)$ admits algebraic descriptions by work of Bifet, De Concini,
Littelman, and Procesi (see \cite{BDP90, LP90}).  

\smallskip

Here we consider the case where $X$ is the wonderful compactification
of a symmetric space $G/K$ of minimal rank, that is, $G$ is
semi-simple of adjoint type and $\rk(G/K) = \rk(G) - \rk(K)$; the main
examples are the groups $G = (G \times G)/ \diag(G)$ and the spaces
$\PSL(2n)/\PSp(2n)$. Moreover, we follow a purely algebraic
approach: we work over an arbitrary algebraically closed field, and
replace the equivariant cohomology ring with the equivariant
intersection ring $A^*_G(X)$ of \cite{EG98} (for wonderful varieties
over the complex numbers, both rings are isomorphic over the
rationals).

\smallskip

We show in Theorem \ref{thm:eq} that the pull-back map 
$$
r : A^*_G(X) \to A^*_T(Y)^{W_K}
$$
is an isomorphism over the rationals. Here $T \subset G$ denotes a
maximal torus containing a maximal torus $T_K \subset K$ with Weyl
group $W_K$, and $Y$ denotes the closure in $X$ of 
$T/T_K \subset G/K$, so that $Y$ is the toric variety associated
with the Weyl chambers of the restricted root system of $G/K$.

\smallskip

We also determine the images under $r$ of the equivariant Chern
classes of the tangent bundle $T_X$ and its logarithmic analogue
$S_X$. For this, we show in Theorem \ref{thm:split} 
that the normal bundle $N_{Y/X}$ decomposes (as a $T$-linearized
bundle) into a direct sum of line bundles indexed by certain roots of
$K$; moreover, any such line bundle is the pull-back of
$O_{{\bbP}^1}(1)$ under a certain $T$-equivariant morphism 
$Y \to \bbP^1$. By Proposition \ref{prop:prod}, the product of these
morphisms yields a closed immersion of the toric variety $Y$ into a
product of projective lines, indexed by the restricted roots. 

\smallskip

In the case of regular compactifications of reductive groups, Theorem 
\ref{thm:eq} is due to Littelmann and Procesi for equivariant
cohomology rings (see \cite{LP90}); it has been adapted to equivariant
Chow ring in \cite{Br98}. Here, as in the latter paper, we rely on a  
precise version of the localization theorem in equivariant
intersection theory inspired, in turn, by a similar result in
equivariant cohomology, see \cite{GKM99}. The main ingredient is the
finiteness of $T$-stable points and curves in $X$; this also plays an
essential role in Tchoudjem's description of cohomology groups of line
bundles on wonderful varieties of minimal rank, see \cite{Tc05}.

\smallskip

For wonderful group compactifications, a more precise, ``additive''
description of the equivariant cohomology ring is due to Strickland,
see \cite{St06}; an analogous description of the equivariant 
Grothendieck group has been obtained by Uma in \cite{Um05}. Both 
results may be generalized to our setting of minimal rank. However,
determining generators and relations for the equivariant cohomology 
or Grothendieck ring is still an open question; see \cite{Br04,Um05} 
for some steps in this direction.

\smallskip

Our determination of the equivariant Chern classes seems to be new, 
already in the group case; it yields a closed formula for the image 
under $r$ of the equivariant Todd class of $X$, analogous to the 
well-known formula expressing the Todd class of a toric variety in 
terms of boundary divisors. The toric variety $Y$ associated to Weyl 
chambers is considered in \cite{Pr90,DL94}, where its cohomology is 
described as a graded representation of the Weyl group; its
realization as a general orbit closure in a product of projective 
lines seems to have been unnoticed. 

\smallskip

Our results extend readily to all regular compactifications of 
symmetric spaces of minimal rank. Specifically, the description of 
the equivariant Chow ring holds unchanged, with a similar proof, 
and the determination of equivariant Chern classes follows from the
wonderful case by the results of \cite[Sec.~5]{Ki06}. Another direct 
generalization concerns the spherical (not necessarily symmetric) 
varieties of minimal rank considered in \cite{Tc05}. Indeed, the 
structure of such varieties may be reduced to the symmetric case, 
as shown by Ressayre in \cite{Re04}.

\smallskip

This article is organized as follows. Section 1 gathers preliminary
notions and results on symmetric spaces, their wonderful
compactifications, and the associated toric varieties. In particular,
for a symmetric space $G/K$ of minimal rank, we study the relations
between the root systems and Weyl groups of $G$, $K$, and $G/K$; these
are our main combinatorial tools. In Section 2, we first describe the
$T$-stable points and curves in a wonderful symmetric variety $X$
of minimal rank; then we obtain our main structure result for
$A^*_G(X)$, and some useful complements as well. Section 3 contains
the decompositions of $N_{Y/X}$ and of the restrictions $T_X\vert_Y$,
$S_X \vert_Y$, together with their applications to equivariant Chern
and Todd classes.

\smallskip

Throughout this article, we consider algebraic varieties over an
algebraically closed field $k$ of characteristic $\neq 2$; by a point
of such a variety, we mean a closed point. As general references, we
use \cite{Ha77} for algebraic geometry, and \cite{Sp98} for algebraic
groups.

\section{Preliminaries}

\subsection{The toric variety associated with Weyl chambers}
\label{subsec:tv}

Let $\Phi$ be a root system in a real vector space $V$ (we follow the
conventions of \cite{Bo81} for root systems; in particular, $\Phi$ is
finite but not necessarily reduced). Let $W$ be the Weyl group, $Q$
the root lattice in $V$, and $Q^{\vee}$ the dual lattice (the
co-weight lattice) in the dual vector space $V^*$. The Weyl chambers
form a subdivision of $V^*$ into rational polyhedral convex cones; let
$\Sigma$ be the fan of $V^*$ consisting of all Weyl chambers and their
faces. The pair $(Q^{\vee},\Sigma)$ corresponds to a toric variety
$$
Y = Y(\Phi)
$$ 
equipped with an action of $W$ via its action on $Q^{\vee}$ which
permutes the Weyl chambers. The group $W$ acts compatibly on the
associated torus
$$
T := \Hom(Q,\bbG_m) = Q^{\vee} \otimes_{\bbZ} \bbG_m.
$$
Thus, $Y$ is equipped with an action of the semi-direct product
$T \, W$. Note that the character group $\cX(T)$ is identified with
$Q$; in particular, we may regard each root $\alpha$ as a
homomorphism 
$$
\alpha: T \to \bbG_m.
$$ 

The choice of a basis of $\Phi$, 
$$
\Delta = \{\alpha_1,\ldots,\alpha_r\},
$$ 
defines a positive Weyl chamber, the dual cone to $\Delta$. Let 
$Y_0 \subset Y$ be the corresponding $T$-stable open affine subset. 
Then $Y_0$ is isomorphic to the affine space $\bbA^r$ on which $T$ acts 
linearly with weights $-\alpha_1,\ldots,-\alpha_r$. Moreover,
the translates $w \cdot Y_0$, where $w \in W$, form an open covering
of $Y$.

In particular, the variety $Y$ is nonsingular. Also, $Y$ is
projective, as $\Sigma$ is the normal fan to the convex polytope with
vertices $w \cdot v$ ($w \in W$), where $v$ is any prescribed regular
element of $V$. The following result yields an explicit projective
embedding of~$Y$:

\begin{proposition}\label{prop:prod}
{\rm (i)} For any $\alpha \in \Phi$, the morphism 
$\alpha : T \to \bbG_m$ extends to a morphism
$$
f_{\alpha} : Y \to \bbP^1.
$$
Moreover, $f_{\alpha}$ and $f_{-\alpha}$ differ by the inverse map
$\bbP^1 \to \bbP^1$, $z \mapsto z^{-1}$.
\smallskip

\noindent
{\rm (ii)} The product morphism
$$
f := \prod_{\alpha \in \Phi} f_{\alpha} : Y \to 
\prod_{\alpha \in \Phi} \bbP^1
$$
is a closed immersion. It is equivariant under $T \, W$,
where $T$ acts on the right-hand side via its action on each factor
$\bbP^1_{\alpha}$ through the character $\alpha$, and $W$ acts via its
natural action on the set $\Phi$ of indices.

\smallskip

\noindent
{\rm (iii)} Conversely, the $T$-orbit closure of any point of 
$\prod_{\alpha \in \Phi} (\bbP^1 \setminus \{0,\infty \})$ 
is isomorphic to $Y$. 

\smallskip

\noindent
{\rm (iv)} Any non-constant morphism $F : Y \to C$, where $C$ is an
irreducible curve, factors through $f_\alpha : Y \to \bbP^1$ where
$\alpha$ is an indivisible root, unique up to sign. Moreover,
\begin{equation}\label{eqn:fxp}
(f_\alpha)_* \cO_Y = \cO_{\bbP^1}.
\end{equation}

\end{proposition}

\begin{proof}
(i) Since $\alpha$ has a constant sign on each Weyl chamber, it
defines a morphism of fans from $\Sigma$ to the fan of $\bbP^1$,
consisting of two opposite half-lines and the origin. This implies
our statement.

(ii) The equivariance property of $f$ is readily verified. Moreover,
the product map
$$
\prod_{i=1}^r f_{\alpha_i} : Y \to (\bbP^1)^r
$$
restricts to an isomorphism 
$Y_0 \to (\bbP^1 \setminus \{\infty\})^r$, 
since each $f_{\alpha_i}$ restricts to the $i$-th coordinate function
on $Y_0 \cong \bbA^r$. Since $Y = W \cdot Y_0$, it follows that $f$ is
a closed immersion.

(iii) follows from (ii) by using the action of tuples
$(t_\alpha)_{\alpha \in  \Phi}$ of non-zero scalars, via component-wise
multiplication. 

(iv) Taking the Stein factorization, we may assume that 
$F_* \cO_Y = \cO_C$. Then $C$ is normal, and hence nonsingular.
Moreover, the action of $T$ on $Y$ descends to a unique action on $C$ 
such that $F$ is equivariant (indeed, $F$ equals the canonical
morphism $Y \to \Proj R(Y,F^* \cL)$, where $\cL$ is any ample
invertible sheaf on $C$, and $R(Y, F^*\cL)$ denotes the section ring
$\bigoplus_n \Gamma(Y,F^*\cL^n)$. Furthermore, $F^*\cL$ admits a 
$T$-linearization, and hence $T$ acts on $R(Y,\cL)$). It follows that
$C \cong \bbP^1$ where $T$ acts through a character $\chi$, uniquely
defined up to sign. Thus, $F$ induces a morphism from the fan of $Y$ to 
the fan of $\bbP^1$; this morphism is given by the linear map 
$\chi : V^* \to \bbR$. In other words, $\chi$ has a constant sign on
each Weyl chamber. Thus, $\chi$ is an integral multiple of an
indivisible root $\alpha$, uniquely defined up to sign. Since $F$ has
connected fibers, then $\chi = \pm \alpha$. 

Conversely, if $\alpha$ is an indivisible root, then the fibers of
the morphism $\alpha :  T \to \bbG_m$ are irreducible. This implies
(\ref{eqn:fxp}). 
\end{proof}

Next, for later use, we determine the divisor of each $f_\alpha$
regarded as a rational function on $Y$. Since $f_\alpha$ is a
$T$-eigenvector, its divisor is a linear combination of the $T$-stable
prime divisors $Y_1,\ldots,Y_m$ of the toric variety $Y$, also called
its \emph{boundary divisors}. Recall that $Y_1,\ldots,Y_m$ correspond
bijectively to the rays of the Weyl chambers, i.e., to the
$W$-translates of the fundamental co-weights
$\omega_1^{\vee},\ldots,\omega_r^{\vee}$ 
(which form the dual basis of the basis of simple roots). The isotropy
group of each $\omega_i^{\vee}$ in $W$ is the maximal parabolic
subgroup $W_i$ generated by the reflections associated with the
simple roots $\alpha_j$, $j \neq i$. Thus, the orbit 
$W \omega_i^{\vee} \cong W/W_i$ is in bijection with the subset
$$
W^i := \{ w\in W ~\vert~ w\alpha_j \in \Phi^+ \text{ for all }j \neq i \}
$$
of minimal representatives for the coset space $W/W_i$. So the
boundary divisors are indexed by the set
$$
E(\Phi) := \{(i,w) ~\vert~ 1 \leq i \leq r, ~ w \in W^i\}
\cong \{ w\omega_i^{\vee} ~\vert~ 1 \leq i \leq r, ~ w \in W\};
$$
we will denote these divisors by $Y_{i,w}$. Furthermore, we have 
\begin{equation}\label{eqn:div}
\divi(f_\alpha) = \sum_{(i,w) \in E(\Phi)} 
\langle \alpha, w \omega_i^{\vee} \rangle \;Y_{i,w}
\end{equation}
by Proposition \ref{prop:prod} and the classical formula
for the divisor of a character in a toric variety (see e.g. 
\cite[Prop.~2.1]{Od88}). Also, note that 
$\langle \alpha, w \omega_i^{\vee} \rangle$
is the $i$-th coordinate of $w^{-1}\alpha$ in the basis of simple
roots.

\subsection{Symmetric spaces}
\label{subsec:ss}

Let $G$ be a connected reductive algebraic group, and 
$$
\theta: G \to G
$$ 
an involutive automorphism. Denote by 
$$
K = G^{\theta} \subset G
$$ 
the subgroup of fixed points; then the homogeneous space $G/K$ is a
\emph{symmetric space}. 

We now collect some results on the structure of symmetric spaces,
referring to \cite{Ri82, Sp85} for details and proofs.
The identity component $K^0$ is reductive, and non-trivial unless $G$
is a $\theta$-\emph{split torus}, i.e., a torus where $\theta$ acts
via the inverse map $g \mapsto g^{-1}$.
 
A parabolic subgroup $P \subseteq G$ is said to be
$\theta$-\emph{split} if the parabolic subgroup $\theta(P)$ is
opposite to $P$. The minimal $\theta$-split parabolic subgroups are
all conjugate by elements of $K^0$; we choose such a subgroup $P$ and
put
$$
L := P \cap \theta(P),
$$
a $\theta$-stable Levi subgroup of $P$. The intersection $L \cap K$
contains the derived subgroup $[L,L]$; thus, every maximal torus of
$L$ is $\theta$-stable. We choose such a torus $T$, so that
\begin{equation}\label{eqn:dec}
T = T^{\theta} T^{-\theta} \quad \text{and} \quad
T^{\theta} \cap T^{-\theta} \text{ is finite.}
\end{equation}
Moreover, the identity component 
$$
A := T^{-\theta,0}
$$ 
is a maximal $\theta$-split subtorus of $G$. All such subtori are
conjugate in $K^0$; their common dimension is the \emph{rank} of the
symmetric space $G/K$, denoted by $\rk(G/K)$. Moreover,
$$
C_G(A) = L = (L \cap K) A
$$
(where $C_G(A)$ denotes the centralizer of $A$ in $G$), and 
$(L \cap K) \cap A = A \cap K$ consists of all elements of order $2$ of
$A$.  

The product $P K^0 \subseteq G$ is open, and equals $PK$; thus, $PK/K$
is an open subset of $G/K$, isomorphic to $P/P\cap K = P/L \cap K$. 
Let $P_u$ be the unipotent radical of $P$, so that 
$P = P_u L$. Then the map
\begin{equation}\label{eqn:id}
\iota : P_u \times A/A \cap K \to PK/K, \quad (g,x) \mapsto g \cdot x
\end{equation}
is an isomorphism.

The character group $\cX(A/A \cap K)$ may be identified with the
subgroup $2 \cX(A)\subset \cX(A)$. On the other hand, 
$A/A \cap K \cong T/T \cap K$ and hence 
$\cX(A/A \cap K)$ may be identified with the subgroup of $\cX(T)$
consisting of those characters that vanish on $T \cap K$, i.e., 
\begin{equation}\label{eqn:char}
\cX(A/A \cap K) = \{ \chi - \theta(\chi) ~\vert~ \chi \in \cX(T)\}.
\end{equation}
Here $\theta$ acts on $\cX(T)$ via its action on $T$. 

Denote by 
$$
\Phi_G \subset \cX(T)
$$ 
the root system of $(G,T)$, with Weyl group 
$$
W_G = N_G(T)/T.
$$ 
Choose a basis $\Delta_G$ consisting of roots of $P$. Let  
$\Phi^+_G \subset \Phi_G$ be the corresponding subset of
positive roots and let $\Delta_L \subset \Delta_G$ be the subset of
simple roots of $L$. The natural action of the involution $\theta$ on
$\Phi$ fixes point-wise the subroot system $\Phi_L$. Moreover,
$\theta$  exchanges the subsets $\Phi^+_G \setminus \Phi^+_L$ and
$\Phi^-_G \setminus \Phi^-_L$ (the sets of roots of $P_u$ and of
$\theta(P_u) = \theta(P)_u$).

Also, denote by 
$$
p : \cX(T) \to \cX(A)
$$ 
the restriction map from the character group of $T$ to that of
$A$. Then $p(\Phi_G) \setminus \{0\}$ is a (possibly non-reduced) root 
system called the \emph{restricted root system}, that we denote by
$\Phi_{G/K}$. Moreover, 
$$
\Delta_{G/K} := p(\Delta_G \setminus \Delta_L)
$$ 
is a basis of $\Phi_{G/K}$. The corresponding Weyl group is 
\begin{equation}\label{eqn:wg}
W_{G/K} = N_G(A)/C_G(A) \cong N_{K^0}(A)/C_{K^0}(A).
\end{equation}
Also, $W_{G/K} \cong N_W(A)/C_W(A)$, and $N_W(A) = W_G^{\theta}$
whereas $C_W(A) = W_L$. This yields an exact sequence
\begin{equation}\label{eqn:ext}
1 \to W_L \to W_G^{\theta} \to W_{G/K} \to 1.
\end{equation}

\subsection{The wonderful compactification of an adjoint symmetric space} 
\label{subsec:wc}

We keep the notation and assumptions of Subsec.~\ref{subsec:ss} and we
assume, in addition, that $G$ is semi-simple and adjoint;
equivalently, $\Delta_G$ is a basis of $\cX(T)$. Then the symmetric
space $G/K$ is said to be \emph{adjoint} as well.

By \cite{DP83, DS99}, $G/K$ admits a canonical compactification: the
\emph{wonderful compactification} $X$, which satisfies the following
properties.

\noindent
(i) $X$ is a nonsingular projective variety.

\noindent
(ii) $G$ acts on $X$ with an open orbit isomorphic to $G/K$.

\noindent
(iii) The complement of the open orbit is the union of 
$r = \rk(G/K)$ nonsingular prime divisors $X_1, \ldots, X_r$ with
normal crossings. 

\noindent
(iv) The $G$-orbit closures in $X$ are exactly the partial
intersections 
$$
X_I := \bigcap_{i \in I} X_i
$$
where $I$ runs over the subsets of $\{1,\ldots,r\}$.

\noindent
(v) The unique closed orbit, $X_1 \cap \cdots \cap X_r$,  
is isomorphic to $G/P \cong G/\theta(P)$.

We say that $X$ is a \emph{wonderful symmetric variety}
with \emph{boundary divisors} $X_1,\ldots,X_r$. By (iii) and (iv),
each orbit closure $X_I$ is nonsingular.

Let $Y$ be the closure in $X$ of the subset 
$$
A/A \cap K \cong AK/K = LK/K \subseteq G/K.
$$ 
Then $Y$ is stable under the action of the subgroup 
$L N_K(A) \subseteq G$. Since $L \cap N_K(A) = L \cap K = C_K(A)$,
and $N_K(A)/C_K(A) \cong W_{G/K}$ by (\ref{eqn:wg}), we obtain an
exact sequence
$$
1 \to L \to L N_K(A) \to W_{G/K} \to 1.
$$
Moreover, since $Y$ is fixed point-wise by $L \cap K$, the action of
$L N_K(A)$ factors through an action of the semi-direct product 
$$
(L/L \cap K) \, W_{G/K} \cong (A/A \cap K) \, W_{G/K}.
$$  

The adjointness of $G$ and (\ref{eqn:char}) imply that 
$\cX(A/A \cap K)$ is the restricted root lattice, with basis
$$
\Delta_{G/K} = \{ \alpha - \theta(\alpha) ~\vert~ 
\alpha \in \Delta_G \setminus \Delta_L \}.
$$
Moreover, $Y$ is the toric variety associated with the Weyl
chambers of the restricted root system $\Phi_{G/K}$ as in
Subsec.~\ref{subsec:tv}. This defines the open affine toric 
subvariety $Y_0 \subset Y$ associated with the positive Weyl
chamber dual to $\Delta_{G/K}$. Note that
\begin{equation}\label{eqn:ywy}
Y = W_{G/K} \cdot Y_0.
\end{equation}

Also, recall the local structure of the wonderful symmetric variety 
$X$: the subset 
$$
X_0 := P \cdot Y_0 = P_u \cdot Y_0
$$ 
is open in $X$, and the map
\begin{equation}\label{eqn:ls}
\iota: P_u \times Y_0 \to X_0, \quad (g,x) \mapsto g \cdot x
\end{equation}
is a $P$-equivariant isomorphism. Moreover, any $G$-orbit in $X$ meets
$X_0$ along a unique orbit of $P$, and meets transversally $Y_0$
along a unique orbit of $A/A\cap K$.

It follows that the $G$-orbit structure of $X$ is determined by that
of the associated toric variety $Y$: any $G$-orbit in $X$ meets
transversally $Y$ along a disjoint union of orbit closures of 
$A/A \cap K$, permuted transitively by $W_{G/K}$. As another
consequence, $X_0 \cap G/K = PK/K$ and $Y_0 \cap G/K = AK/K$, so
that $\iota$ restricts to the isomorphism (\ref{eqn:id}).

Finally, the closed $G$-orbit $X_1 \cap \cdots \cap X_r$ 
meets $Y_0$ transversally at a unique point $z$, the $T$-fixed point
in $Y_0$. The isotropy group $G_z$ equals $\theta(P)$, and the normal
space to $G \cdot z$ at $z$ is identified with the tangent space to
$Y$ at that point. Hence the weights of $T$ in the tangent space to
$X$ at $z$ are the positive roots 
$\alpha \in \Phi^+_G \setminus \Phi^+_L$ 
(the contribution of the tangent space to $G \cdot z$), and the simple 
restricted roots $\gamma = \alpha - \theta(\alpha)$, where 
$\alpha \in \Delta_G \setminus \Delta_L$ 
(the contribution of the tangent space to $Y$).

\subsection{Symmetric spaces of minimal rank}
\label{subsec:ssm}

We return to the setting of Subsect.~\ref{subsec:ss}. In particular,
we consider a connected reductive group $G$ equipped with an
involutive automorphism $\theta$, and the fixed point subgroup 
$K = G^{\theta}$. 

Let $T$ be any $\theta$-stable maximal torus of $G$. Then
(\ref{eqn:dec}) implies that 
$$
\rk(G) \geq \rk(K) + \rk(G/K)
$$ 
with equality if and only if the identity component $T^{\theta,0}$ is
a maximal torus of $K^0$, and $T^{-\theta,0}$ is a maximal
$\theta$-split subtorus. We then say that the symmetric space $G/K$ is
\emph{of minimal rank}; equivalently, all $\theta$-stable maximal tori
of $G$ are conjugate in $K^0$. 

We refer to \cite[Subsec.~3.2]{Br04} for the proof of the following
auxiliary result, where we put
$$
T_K := T^{\theta,0} = (T \cap K)^0.
$$ 

\begin{lemma}\label{lem:ro}
{\rm (i)} The roots of $(K^0, T_K)$ are exactly the restrictions to 
$T_K$ of the roots of $(G,T)$.

\smallskip

\noindent
{\rm (ii)} The Weyl group of $(K^0,T_K)$ may be identified with
$W^{\theta}_G$.
\end{lemma}

In particular, $C_G(T_K) = T$ by (i) (this may also be seen
directly). We put
$$
W_K := W^{\theta}_G \quad \text{and} \quad 
N_K := N_{K^0}(T) = N_{K^0}(T_K).
$$ 
By (ii), this yields an exact sequence  
\begin{equation}\label{eqn:ex}
1 \to T_K  \to N_K \to W_K \to 1.
\end{equation}
Moreover, by (\ref{eqn:ext}), $W_K$ fits into an exact sequence
\begin{equation}\label{eqn:exa}
1 \to W_L \to W_K \to W_{G/K} \to 1.
\end{equation}
The group $W_K$ acts on $\Phi_G$ and stabilizes the subset
$\Phi_G^{\theta} = \Phi_L$; the restriction map
$$
q : \cX(T) \to \cX(T_K)
$$
is $W_K$-equivariant and $\theta$-invariant. Denoting by $\Phi_K$ the
root system of $(K^0,T_K)$, Lemma \ref{lem:ro} (i) yields the equality
$$
\Phi_K = q(\Phi_G).
$$

We now obtain two additional auxiliary results: 

\begin{lemma}\label{lem:fi}
Let $\beta \in \Phi_K$. Then one of the following cases occurs:

\smallskip

\noindent
{\rm (a)} $q^{-1}(\beta)$ consists of a unique root, 
$\alpha \in \Phi_L$. 

\smallskip

\noindent
{\rm (b)} $q^{-1}(\beta)$ consists of two strongly orthogonal roots
$\alpha$, $\theta(\alpha)$, where $\alpha \in \Phi^+_G \setminus \Phi^+_L$
and $\theta(\alpha) \in \Phi^-_G \setminus \Phi^-_L$.

\smallskip

In particular, $q$ induces a bijection 
$q^{-1}q(\Phi_L) = \Phi_L \cong q(\Phi_L)$. Moreover, 
$\alpha$ and $\theta(\alpha)$ are strongly orthogonal for any 
$\alpha \in \Phi_G \setminus \Phi_L$; then 
$s_\alpha s_{\theta(\alpha)} \in W_G^{\theta} = W_K$ is a
representative of the reflection of $W_{G/K}$ associated with the
restricted root $\alpha - \theta(\alpha)$.  
\end{lemma}

\begin{proof}
Note that any $\alpha \in q^{-1}(\beta)$ is a root of $(C_G(S),T)$,
where $S \subseteq T_K$ denotes the identity component of the kernel
of $\beta$, and $C_G(S)$ stands for the centralizer of $S$ in
$G$. Also, $C_G(S)$ is a connected reductive $\theta$-stable
subgroup of $G$, and the symmetric space $C_G(S)/C_K(S)$ is of minimal
rank. Moreover, $\theta$ yields an involution of the quotient group
$C_G(S)/S$, and the corresponding symmetric space is still of minimal
rank. So we may reduce to the case where $S$ is trivial, i.e., $K$ has
rank $1$. Since $\beta$ is a root of $K^0$, it follows that 
$K^0 \cong \SL(2)$ or $\PSL(2)$. Together with the minimal rank
assumption, it follows that one of the following cases occurs, up to
an isogeny of $G$:

\noindent
(a) $K^0 = G = \PSL(2)$; then $\theta$ is trivial.

\noindent
(b) $K^0 = \PSL(2)$ and $G = \PSL(2) \times \PSL(2)$; then $\theta$
exchanges both factors.

This implies our assertions.
\end{proof} 

We will identify $q(\Phi_L)$ with $\Phi_L$ in view of Lemma
\ref{lem:fi}.

\begin{lemma}\label{lem:res}
{\rm (i)} $q$ induces bijections 
$$
\Phi_G^+ \setminus \Phi_L^+ \cong \Phi_K \setminus \Phi_L 
\cong \Phi_G^- \setminus \Phi_L^- . 
$$

\smallskip

\noindent
{\rm (ii)} $\Phi_K \setminus \Phi_L$ is a root system, stable under
the action of $W_K$ on $\Phi_K$.

\smallskip

\noindent
{\rm (iii)} The restricted root system $\Phi_{G/K}$ is reduced. 
\end{lemma}

\begin{proof}
(i) follows readily from Lemma \ref{lem:fi}.

(ii) Since $\Phi_L$ is stable under $W_K$, then so is 
$\Phi_K \setminus \Phi_L$. In particular, the latter is stable
under any reflection $s_\beta$, where
$\beta \in \Phi_K \setminus \Phi_L$. It follows that 
$\Phi_K \setminus \Phi_L$ is a root system.

(iii) We have to check the non-existence of roots 
$\alpha_1, \alpha_2 \in \Phi_G \setminus \Phi_L$ such that 
$\alpha_2 - \theta(\alpha_2) = 2(\alpha_1 - \theta(\alpha_1))$. 
Considering the identity component of the intersection of kernels of
$\alpha_1$, $\theta(\alpha_1)$ and $\alpha_2$ and arguing as in the
proof of Lemma \ref{lem:fi}, we may assume that 
$\rk(G) \leq 3$. Clearly, we may further assume that $G$ is 
semi-simple and adjoint. Then the pair $(G,K)$ is either
$(\PSL(2) \times \PSL(2), \PSL(2))$ or $(\PSL(4),\PSp(4))$, and the
assertion follows by inspection.
\end{proof}

Also, note that the pull-back under $q$ of any system of positive roots
of $\Phi_K$ is a $\theta$-stable system of positive roots of
$\Phi_G$. Let $\Sigma_G$ be the corresponding basis of $\Phi_G$; then
$q(\Sigma_G)$ is a basis of $\Phi_K$. If $G$ is semi-simple, then the
involution $\theta$ is uniquely determined by the its restriction to
$\Sigma_G$, and the latter is an involution of the Dynkin diagram of
$G$. It follows that the symmetric spaces of minimal rank under an
adjoint semi-simple group are exactly the products of symmetric spaces
that occur in the following list.

\begin{examples}
1) Let $G = \bG \times \bG$, where $\bG$ is an adjoint semi-simple
group, and let $\theta$ be the involution of $G$ such that
$\theta(x,y) = (y,x)$; then $K = \diag(\bG)$ and $\rk(G/K) = \rk(\bG)$.
 
The maximal $\theta$-stable subtori $T \subset G$ are exactly the
products  $\bT \times \bT$, where $\bT$ is a maximal torus of $\bG$;
then $T_K = \diag(\bT)$, and $A = \{(x,x^{-1}) ~\vert~ x \in \bT \}$.
Thus, $L = T$, and 
$W_K = W_{G/K} = \diag(W_{\bG}) \subset W_{\bG} \times W_{\bG} = W_G$.

Moreover, $\Phi_G$ is the disjoint union of two copies of
$\Phi_{\bG}$, each of them being mapped isomorphically to 
$\Phi_K = \Phi_{G/K}$ by $q$.
\smallskip

\noindent
2) Consider the group $G = \PSL(2n)$ and the involution $\theta$ 
associated with the symmetry of the Dynkin diagram; then 
$K = \PSp(2n)$ and $\rk(G/K) = n-1$.

The Levi subgroup $L$ is the image in $G = \PGL(2n)$ of the product
$\GL(2) \times \cdots \times \GL(2)$ ($n$ copies). The Weyl group
$W_G$ is the symmetric group $S_{2n}$, and $W_K$ is the subgroup
preserving the partition of the set $\{1, 2, \ldots,2n \}$ into the
$n$ subsets $\{1,2 \}$, $\{3,4 \}$, $\ldots$, $\{ 2n-1,2n \}$. Thus,
$W_K$ is the semi-direct product of $S_2 \times \cdots \times S_2$
($n$ copies) with $S_n$. Moreover, 
$W_L = S_2 \times \cdots \times S_2$, so that $W_{G/K} = S_n$. 

For the root systems, we have $\Phi_G = A_{2n-1}$, $\Phi_K = C_n$,
$\Phi_L = A_1 \times \cdots \times A_1$ (the subset of 
long roots of $\Phi_K$), $\Phi_K \setminus \Phi_L = D_n$ (the short
roots), and $\Phi_{G/K} = A_{n-1}$.

\smallskip

\noindent
3) Consider the group $G = \PSO(2n)$ and the involution $\theta$ 
associated with the symmetry of the Dynkin diagram (of type $D_n$);
then $K = \PSO(2n-1)$ and $\rk(G/K) = 1$. 

The Levi subgroup $L$ is the
image in $G$ of $\SO(2) \times \SO(2n-2)$. The Weyl group
$W_G$ is the semi-direct product of $\{\pm 1 \}^{n-1}$ with $S_n$, and
$W_K$ is the semi-direct product of $\{\pm 1 \}^{n-1}$ with $S_{n-1}$. 
Moreover, $W_L$ is the semi-direct product of $\{\pm 1 \}^{n-2}$ with
$S_{n-1}$, so that $W_{G/K} = \{ \pm 1 \}$. 

We have $\Phi_G = D_n$, $\Phi_K = B_{n-1}$, $\Phi_L = D_{n-1}$ (the
subset of long roots of $\Phi_K$), 
$\Phi_K \setminus \Phi_L = A_1 \times \cdots \times A_1$ ($n-1$
copies), and $\Phi_{G/K} = A_1$.

\smallskip
4) Let $G$ be an adjoint simple group of type $E_6$, and $\theta$ the 
involution associated with the symmetry of the Dynkin diagram; then
$K$ is a simple group of type $F_4$, and $\rk(G/K) = 2$. 

The Levi subgroup $L$ has type $D_4$, and $W_{G/K} = S_3$. We have
$\Phi_G = E_6$, $\Phi_K = F_4$, $\Phi_L = D_4$ (the subset of long
roots of $\Phi_K$), $\Phi_K \setminus \Phi_L = D_4$, and 
$\Phi_{G/K} = A_2$.
\end{examples}

\section{Equivariant Chow ring} 
\label{sec:gws}

\subsection{Wonderful symmetric varieties of minimal rank}
\label{subsec:fpc}

From now on, we consider an adjoint semi-simple group $G$
equipped with an involutive automorphism $\theta$ such that the
corresponding symmetric space $G/K$ is of minimal rank. Then the group
$K$ is connected, semi-simple and adjoint; see \cite[Lem.~5]{Br04}.

We choose a $\theta$-stable maximal torus $T \subseteq G$, so that 
$A := T^{-\theta,0}$ is a maximal $\theta$-split subtorus. Also, we
put $T_K := T^{\theta}$; this group is connected by 
\cite[Lem.~5]{Br04} again. Thus, $T_K$ is a maximal torus of $K$. In
agreement with the notation of Subsec.~\ref{subsec:ssm}, we denote by
$N_K$ the normalizer of $T_K$ in $K$, and by $W_K$ the Weyl group
of $(K,T_K)$.

As in Subsec.~\ref{subsec:wc}, we denote by $X$ the wonderful
compactification of $G/K$, also called a 
\emph{wonderful symmetric variety of minimal rank}. 
The associated toric variety $Y$ is the closure in 
$X$ of $T/T_K \cong A/A \cap K$. Recall that $Y$ is stable under
the subgroup $L N_K \subseteq G$, and fixed point-wise by $L \cap K$. 
Thus, $L N_K$ acts on $Y$ via its quotient group 
$$
L N_K / (L \cap K) \cong (T/T_K) \, W_{G/K} \cong T N_K/T_K.
$$ 
We will mostly consider $Y$ as a $TN_K$-variety.

By \cite[Sec.~10]{Tc05}, $X$ contains only finitely many $T$-stable
curves. We now obtain a precise description of all these curves, and of
those that lie in $Y$. This may be deduced from the results of
[loc.~cit.], which hold in the more general setting of wonderful
varieties of minimal rank, but we prefer to provide direct, somewhat
simpler arguments. 

\begin{lemma}\label{lem:fix}
{\rm (i)} The $T$-fixed points in $X$ (resp.~$Y$) are exactly the
points $w \cdot z$, where $w \in W$ (resp.~$W_K$), and $z$ denotes the
unique $T$-fixed point of $Y_0$. These fixed points are
parametrized by $W_G/W_L$ (resp.~$W_K/W_L \cong W_{G/K}$).

\smallskip

\noindent
{\rm (ii)} For any $\alpha \in \Phi^+_G \setminus \Phi^+_L$, there
exists a unique irreducible $T$-stable curve $C_{z,\alpha}$ which
contains $z$ and on which $T$ acts through its character 
$\alpha$. The $T$-fixed points in $C_{z,\alpha}$ are exactly $z$ and 
$s_\alpha \cdot z$.

\smallskip

\noindent
{\rm (iii)} For any 
$\gamma = \alpha - \theta(\alpha) \in \Delta_{G/K}$, there exists a
unique irreducible $T$-stable curve $C_{z,\gamma}$ which 
contains $z$ and on which $T$ acts through its character
$\gamma$. The $T$-fixed points in $C_{z,\gamma}$ are exactly $z$ and 
$s_\alpha s_{\theta(\alpha)} \cdot z$.

\smallskip

\noindent
{\rm (iv)} The irreducible $T$-stable curves in $X$ are 
the $W_G$-translates of the curves $C_{z,\alpha}$ and $C_{z,\gamma}$.
They are all isomorphic to $\bbP^1$.
\smallskip

\noindent
{\rm (v)} The irreducible $T$-stable curves in $Y$ are the
$W_{G/K}$-translates of the curves $C_{z,\gamma}$.
\end{lemma}

\begin{proof}
The assertions on the $T$-fixed points in $X$ are proved in
\cite[Lem.~6]{Br04}. And since $Y$ is the toric variety associated
with the Weyl chambers of $\Phi_{G/K}$, the group $W_{G/K}$ acts simply
transitively on its $T$-fixed points. This proves (i).

Let $C \subset X$ be an irreducible $T$-stable curve. Replacing $C$
with a $W_G$-translate, we may assume that it contains $z$. Then
$C \cap X_0$ is an irreducible $T$-stable curve in $X_0$, an
affine space where $T$ acts linearly with weights the positive roots 
$\alpha \in \Phi^+_G \setminus \Phi^+_L$, and the simple restricted
roots $\gamma = \alpha - \theta(\alpha)$, 
$\alpha \in \Delta_G \setminus \Delta_L$. Since all these weights have
multiplicity $1$, it follows that $C \cap X_0$ is a coordinate line
in $X_0$. Thus, $C$ is isomorphic to $\bbP^1$ where $T$ acts through
$\alpha$ or $\gamma$. In the former case, $C$ is contained in the
closed $G$-orbit $G \cdot z$; it follows that its other $T$-fixed
point is $s_\alpha \cdot z$. In the latter case, $C$ is contained in
$Y$, and hence its other $T$-fixed point corresponds to a simple
reflection in $W_{G/K}$. By considering the weight of the $T$-action
on $C$, this simple reflection must be the image in $W_{G/K}$ of 
$s_{\alpha} s_{\theta(\alpha)} \in W_K$. 
This implies the remaining assertions (ii)-(v).
\end{proof}

\subsection{Structure of the equivariant Chow ring}
\label{subsec:eq}

We will obtain a description of the $G$-equivariant Chow ring of $X$
with rational coefficients. For this, we briefly recall some
properties of equivariant intersection theory, referring to
\cite{Br97, EG98} for details.

To any nonsingular variety $Z$ carrying an action of a linear
algebraic group $H$, one associates the equivariant Chow ring
$A^*_H(Z)$. This is a positively graded ring with degree-$0$ part
$\bbZ$, and degree-$1$ part the equivariant Picard group $\Pic_H(Z)$
consisting of isomorphism classes of $H$-linearized invertible
sheaves on $Z$. 

Every closed $H$-stable subvariety $Y \subseteq Z$ of codimension
$n$ yields an equivariant class
$$
[Y]_H \in A^n_H(Z).
$$
The class $[Z]_H$ is the unit element of $A^*_H(Z)$.

Any equivariant morphism $f: Z \to Z'$, where $Z'$ is a nonsingular
$H$-variety, yields a pull-back homomorphism
$$
f^* : A^*_H(Z') \to A^*_H(Z).
$$ 
In particular, $A^*_H(Z)$ is an algebra over $A^*_H(pt)$, where $pt$
denotes $\Spec k$.

The equivariant Chow ring of $Z$ is related to the ordinary Chow ring
$A^*(Z)$ via a homomorphism of graded rings
$$
\varphi_H: A^*_H(Z) \to A^*(Z)
$$
which restricts trivially to the ideal of $A^*_H(Z)$ generated by
$A^+_H(pt)$ (the positive part of $A^*_H(pt)$). If $H$ is connected,
then $\varphi_H$ induces an isomorphism over the rationals:
$$
A^*_H(Z)_{\bbQ}/ A^+_H(pt) A^*_H(Z)_{\bbQ} \cong A^*(Z)_{\bbQ}.
$$

More generally, there is a natural homomorphism of graded rings
$$
\varphi_H^{H'}: A^*_H(Z) \to A^*_{H'}(Z)
$$ 
for any closed subgroup $H' \subset H$. If $H' = H^0$, the neutral
component of $H$, then the group of components $H/H^0$ acts on the
graded ring $A^*_{H^0}(Z)$, and the image of $\varphi_H^{H'}$ is
contained in the invariant subring $A^*_{H^0}(Z)^{H/H^0}$. Moreover,
$\varphi_H^{H^0}$ induces an isomorphism of rational equivariant Chow
rings
$$
A^*_H(Z)_{\bbQ} \cong A^*_{H^0}(Z)^{H/H^0}_{\bbQ}.
$$

If $H$ is a connected reductive group, and $T \subseteq H$ is a
maximal torus with normalizer $N$ and associated Weyl group $W$, then
the composite of the canonical maps
$$
A^*_H(Z) \to A^*_N(Z) \to A^*_T(Z)^W 
$$ 
is an isomorphism over the rationals. In particular, we obtain an
isomorphism 
$$
A^*_H(pt)_{\bbQ} \cong A^*_T(pt)_{\bbQ}^W.
$$ 
Furthermore, $A^*_T(pt)$ is canonically isomorphic to the symmetric
algebra (over the integers) of the character group $\cX(T)$. This
algebra will be denoted by $S_T$, or just by $S$ if this yields no
confusion.

Returning to the $G$-variety $X$, we may now state our structure result:

\begin{theorem}\label{thm:eq}
The map
\begin{equation}\label{eqn:iso}
r: A^*_G(X) \to A^*_T(X)^{W_G}  \to  A^*_T(X)^{W_K} \to
A^*_T(Y)^{W_K}
\end{equation}
obtained by composing the canonical maps, is an isomorphism over the
rationals.
\end{theorem}

\begin{proof}
We adapt the arguments of \cite[Sec.~3.1]{Br98} regarding regular
compactifications of reductive groups; our starting point is
the precise version of the localization theorem obtained in
\cite[Sec.~3.4]{Br97}. Together with Lemma \ref{lem:fix},
it implies that the $T$-equivariant Chow ring $A^*_T(X)$ may be
identified as an $S$-algebra to the space of tuples 
$(f_{w \cdot  z})_{w \in W_G/W_L}$ of elements of $S$ such that 
$$
f_{v\cdot z} \equiv f_{w \cdot z} \quad (\modu \chi)
$$
whenever the $T$-fixed points $v \cdot z$ and $w \cdot z$ are joined
by an irreducible $T$-stable curve where $T$ acts through its
character $\chi$. This identification is obtained by restricting to
the fixed points. The ring structure on the above space of tuples is
given by pointwise addition and multiplication; moreover, $S$ is
identified with the subring of constant tuples $(f)$.

It follows that $A^*_G(X)_{\bbQ} \cong A^*_T(X)_{\bbQ}^{W_G}$ may be
identified, via restriction to $z$, with the subring of $S^{W_L}_{\bbQ}$
consisting of those $f$ such that
\begin{equation}\label{eqn:cong}
v^{-1} \cdot f \equiv w^{-1} \cdot f \quad (\modu \chi)
\end{equation}
for all $v$, $w$ and $\chi$ as above. By Lemma \ref{lem:fix} again, it
suffices to check the congruences (\ref{eqn:cong}) when $v = 1$. Then
either we are in case (ii) of that lemma, and $w = s_\alpha$, or we
are in case (iii) and $w = s_{\alpha} s_{\theta(\alpha)}$. In the
former case, (\ref{eqn:cong}) is equivalent to the congruence
$$
f \equiv s_{\alpha} \cdot f \quad (\modu \alpha),
$$
which holds for any $f \in S_{\bbQ}$. In the latter case, we obtain
\begin{equation}\label{eqn:scong}
f \equiv s_{\alpha} s_{\theta(\alpha)} \cdot f \quad 
(\modu \alpha - \theta(\alpha)).
\end{equation}
Thus, $A^*_G(X)_{\bbQ}$ is identified with the subring of $S^{W_L}_{\bbQ}$
defined by the congruences (\ref{eqn:scong}) for all 
$\alpha \in \Delta_G \setminus \Delta_L$.

On the other hand, we may apply the same localization theorem to the
$T$-variety $Y$. Taking invariants of $W_K$ and using the exact
sequence (\ref{eqn:exa}), we see that $A^*_T(Y)^{W_K}_{\bbQ}$ may be
identified with the same subring of $S_{\bbQ}$, by restricting to the
same point $z$. This implies our statement.
\end{proof}

\subsection{Further developments}
\label{subsec:fd}

We will obtain a more precise description of the ring
$A^*_T(Y)^{W_K}_{\bbQ}$ that occurs in Theorem \ref{thm:eq}. This
will not be used in the sequel of this article, but has its own
interest.

\begin{proposition}\label{prop:dec}
{\rm (i)} We have compatible isomorphisms of graded rings
$$
A^*_T(Y)^{W_K}_{\bbQ} \cong (S_{T_K}^{W_L} \otimes 
A^*_{T/T_K}(Y))^{W_{G/K}}_{\bbQ}
$$
and
$$
A^*_T(pt)^{W_K}_{\bbQ} \cong S^{W_K}_{T,\bbQ}
\cong (S_{T_K}^{W_L} \otimes S_{T/T_K})^{W_{G/K}}_{\bbQ}.
$$

\smallskip

\noindent
{\rm (ii)} The image in $A^*_G(X)_{\bbQ} \cong A^*_T(Y)^{W_K}_{\bbQ}$
of the subring
$$
A^*_K(pt)_{\bbQ} \cong S_{T_K,\bbQ}^{W_K} \cong 
(S_{T_K}^{W_L} \otimes \bbQ)^{W_{G/K}}
\subseteq A^*_T(pt)^{W_K}_{\bbQ}
$$
is mapped isomorphically to $A^*_G(G/K)_{\bbQ} \cong A^*_K(pt)_{\bbQ}$
under the pull-back from $X$ to the open orbit $G/K$.

\smallskip

\noindent
{\rm (iii)} We have isomorphisms
\begin{equation}\label{eqn:is}
\Pic(X)_{\bbQ} \cong \Pic_G(X)_{\bbQ} \cong 
\Pic_T(Y)^{W_K}_{\bbQ} \cong \Pic_{T/T_K}(Y)^{W_K}_{\bbQ}
\end{equation}
that identify the class $[X_i]$ of any boundary divisor with 
\begin{equation}\label{eqn:sum}
[X_i \cap Y]_{T/T_K} = \sum_{w \in W_{G/K}^i} [Y_{i,w}]_{T/T_K}
\end{equation}
where $Y_{i,w}$ denote the boundary divisors of $Y$, indexed as in
Subsec.~\ref{subsec:tv}.
\end{proposition}

\begin{proof}
(i) Lemma \ref{lem:rel} below yields a $W_K$-equivariant isomorphism
of graded $S_T$-algebras 
$$
A^*_T(Y) \cong S_T \otimes_{S_{T/T_K}} A^*_{T/T_K}(Y),
$$
where $W_K$ acts on $S_T$ via its action on $T$, and on
$A^*_{T/T_K}(Y)$ via its compatible actions on $T/T_K$ and $Y$. Moreover, 
$\cX(T)_{\bbQ} \cong \cX(T_K)_{\bbQ} \oplus \cX(T/T_K)_{\bbQ}$ 
as $W_K$-modules, so that 
$$
S_{T,\bbQ} \cong S_{T_K,\bbQ} \otimes S_{T/T_K, \bbQ}
$$
as graded $W_K$-algebras. It follows that 
$$
A^*_T(Y)_{\bbQ} \cong S_{T_K,\bbQ} \otimes A^*_{T/T_K}(Y)_{\bbQ}
$$
as graded $S_{T,\bbQ}$-$W_K$-algebras. Taking $W_K$-invariants and
observing that the action of $W_L \subseteq W_K$ on the right-hand
side fixes pointwise $A^*_{T/T_K}(Y)_{\bbQ}$, we obtain the desired
isomorphisms in view of the exact sequence (\ref{eqn:exa}).

(ii) Since $(G/K) \cap Y = T/T_K$, we obtain a commutative square
$$
\CD
A^*_G(X) @>{r}>> A^*_T(Y)^{W_K} \\
@VVV @V{t}VV \\
A^*_G(G/K) @>{s}>> A^*_T(T/T_K)^{W_K} \\
\endCD
$$
where the vertical arrows are pull-backs, and $s$ is defined
analogously to $r$. Moreover, $A^*_G(G/K) \cong A^*_K(pt)$, 
$A^*_T(T/T_K) \cong A^*_{T_K}(pt)$, and this identifies $s_{\bbQ}$
with the isomorphism 
$A^*_K(pt)_{\bbQ} \to S_{T_K,\bbQ}^{W_K}$. Likewise, $t_{\bbQ}$ is
identified with the map 
$$
(S_{T_K}^{W_L} \otimes A^*_{T/T_K}(Y))^{W_{G/K}}_{\bbQ}
\to 
(S_{T_K}^{W_L} \otimes \bbQ)^{W_{G/K}}
$$
induced by the natural map $ A^*_{T/T_K}(Y)_{\bbQ} \to \bbQ$. 
This implies our assertion.

(iii) Let $\cL$ be an invertible sheaf on $X$, then some positive
tensor power $\cL^n$ admits a $G$-linearization, and such a 
linearization is unique since $G$ is semi-simple. This implies the
first isomorphism of (\ref{eqn:is}). The second isomorphism is a
consequence of Theorem \ref{thm:eq}. To show the third isomorphism,
recall that $\Pic^T(Y) \cong A^1_T(Y)$, so that
$$
\Pic_T(Y)^{W_K}_{\bbQ} \cong 
(\cX(T_K)^{W_L} \oplus \Pic_{T/T_K}(Y))^{W_{G/K}}_{\bbQ}
\cong 
\cX(T_K)^{W_K}_{\bbQ} \oplus \Pic_{T/T_K}(Y)^{W_{G/K}}_{\bbQ}
$$
by (i); moreover, $\cX(T_K)^{W_K}_{\bbQ} =0$ since the group $K$ is
semi-simple. Finally, (\ref{eqn:sum}) follows from the decomposition 
of $X_i \cap Y$ into irreducible components $Y_{i,w}$, each of
them having intersection multiplicity one. 
\end{proof}

\begin{lemma}\label{lem:rel}
Let $Z$ be a nonsingular variety carrying an action of a torus $T$ and
let $T' \subset T$ be a closed subgroup acting trivially on $Z$. Then 
there is a natural isomorphism of graded $S_T$-algebras
$$
A^*_T(Z) \cong S_T \otimes_{S_{T/T'}} A^*_{T/T'}(Z),
$$
where $S_{T/T'}$ is identified with a subring of $S_T$ via the
inclusion of $\cX(T/T')$ into $\cX(T)$.

In particular, if $T'$ is finite then there is a natural isomorphism
of graded algebras over $S_{T,\bbQ} \cong S_{T/T',\bbQ}$:
$$
A^*_T(Z)_{\bbQ} \cong A^*_{T/T'}(Z)_{\bbQ}.
$$
\end{lemma}

\begin{proof}
We begin by constructing a morphism of graded $S_{T/T'}$-algebras
$$
f: A^*_{T/T'}(Z) \to A^*_T(Z)
$$
such that $f([Y]_{T/T'}) =  [Y]_T$ for any $T$-stable subvariety 
$Y \subseteq Z$. 

For this, we work in a fixed degree $n$ and consider a pair $(V,U)$,
where $V$ is a finite-dimensional $T$-module, $U \subset V$ is a
$T$-stable open subset such that the quotient $U \to U/T$ is a
principal $T$-bundle, and the codimension of $V \setminus U$ is
sufficiently large; see \cite{EG98} for details. Then we can form the
mixed quotient 
$$
Z \times^T U := (Z \times U)/T,
$$ 
and we obtain
\begin{equation}\label{eqn:isom}
A^n_T(Z) = A^n(Z \times^T U) \cong A^n(Z \times^{T/T'} U/T')
\cong A^n_{T/T'}(Z \times U/T'),
\end{equation}
where the latter isomorphism follows from the freeness of the
diagonal $T/T'$-action on $Z \times U/T'$. Composing (\ref{eqn:isom})
with the pull-back under the projection 
$Z \times U/T' \to Z$ yields a morphism
$$
f_n: A^n_{T/T'}(Z) \to A^n_T(Z).
$$
One may check as in \cite{EG98} that $f_n$ is independent of the
choice of $(V,U)$, and hence yields the desired morphism $f$. 

Using the description of the $S_{T/T'}$-module $A^*_{T/T'}(Z)$
and of the $S$-module $A^*_T(Z)$ in terms of invariant cycles
(see \cite[Thm.~2.1]{Br97}), we obtain an isomorphism of graded
$S_T$-modules 
$$
\id \otimes f : S_T \otimes_{S_{T/T'}} A^*_{T/T'}(Z) \to A^*_T(Z).
$$
\end{proof}

Next, we show that $A^*_G(X)_{\bbQ}$ is a free module over a big
polynomial subring:

\begin{proposition}\label{prop:hsp}
Let $R$ denote the $\bbQ$-subalgebra of $A^*_G(X)_{\bbQ}$ generated by
the image of $A^*_K(pt)$ (defined in Proposition \ref{prop:dec}(ii))
and by the equivariant classes $[X_1]_G,\ldots,[X_r]_G$ of the
boundary divisors. Then $R$ is a graded polynomial ring, and the
$R$-module $A^*_G(X)_{\bbQ}$ is free of rank $\vert W_{G/K} \vert$.
\end{proposition}

\begin{proof}
By \cite[6.7 Corollary]{Br97} and Lemma \ref{lem:fix},
$A^*_G(X)_{\bbQ}$ is a free module over $A^*_G(pt)_{\bbQ}$ of rank
being the index $\vert W_G : W_L \vert$. As a consequence, the ring
$A^*_G(X)_{\bbQ}$ is Cohen--Macaulay of dimension $\rk(G)$. 

On the other hand, the $R$-module $A^*_G(X)_{\bbQ}$ is finite by
\cite[Lemma 6]{BP02} (the latter result is proved there in the setting
of equivariant cohomology, but the arguments may  be readily adapted
to equivariant intersection theory). 

Since $R$ is a quotient of a polynomial ring in $\rk(K) + r = \rk(G)$
variables, it follows that $R$ equals this polynomial ring. Moreover,
$A^*_G(X)_{\bbQ}$ is a free $R$-module, since it is
Cohen--Macaulay. This proves all assertions except that on the rank of
the $R$-module $A^*_G(X)_{\bbQ}$, which may be checked by adapting the
Poincar\'e series arguments of \cite{BP02}.
\end{proof}

\section{Equivariant Chern classes}
\label{sec:ecc}

\subsection{The normal bundle of the associated toric variety}
\label{subsec:nb}
We maintain the notation and assumptions of
Subsec.~\ref{subsec:fpc}. In particular, $X$ denotes a wonderful
symmetric variety of minimal rank, with associated toric variety $Y$.

Let $\cN_{Y/X}$ denote the normal sheaf to $Y$ in $X$. This is
a $LN_K$-linearized locally free sheaf on $Y$, which fits into
an exact sequence of such sheaves
\begin{equation}\label{eqn:tn}
0 \to \cT_Y \to \cT_X\vert_Y \to \cN_{Y/X} \to 0.
\end{equation}
Here $\cT_X$ denotes the tangent sheaf to $X$ (this is a
$G$-linearized locally free sheaf on $X$) and $\cT_X\vert_Y$
denotes its pull-back to $Y$.

The action of $G$ on $X$ yields a morphism of $G$-linearized sheaves 
\begin{equation}\label{eqn:op}
op_X: \cO_X \otimes \fg  \to \cT_X,
\end{equation}
where $\fg$ denotes the Lie algebra of $G$. In turn, this yields a
morphism of $T$-linearized sheaves
$\cO_Y \otimes \fg  \to \cN_{Y/X}$
which factors through another such morphism 
\begin{equation}\label{eqn:gg}
\varphi : \cO_Y \otimes \fg/\fl  \to \cN_{Y/X}
\end{equation}
(where $\fl$ denotes the Lie algebra of $L$), since $Y$ is stable
under $L$. Also, note the isomorphism of $T$-modules
$$
\fg/\fl \cong \bigoplus_{\alpha \in \Phi_G \setminus \Phi_L} \fg_{\alpha}.
$$
We may now formulate a splitting theorem for $\cN_{Y/X}$:

\begin{theorem}\label{thm:split}
{\rm (i)} We have a decomposition of $T$-linearized sheaves
\begin{equation}\label{eqn:split}
\cN_{Y/X} = \bigoplus_{\beta \in \Phi_K \setminus \Phi_L} \cL_{\beta}
\end{equation}
where each $\cL_{\beta}$ is an invertible sheaf on which $T_K$
acts via its character $\beta$. The action of $N_K$ on
$\cN_{Y/X}$ permutes the $\cL_{\beta}$'s according to the action
of $W_K$ on $\Phi_K \setminus \Phi_L$.

\noindent
{\rm (ii)} The map (\ref{eqn:gg}) restricts to surjective maps
$$
\varphi_{\beta}: \cO_Y \otimes 
(\fg_{\alpha} \oplus \fg_{\theta(\alpha)}) \to \cL_{\beta} \quad  
(\beta =q(\alpha), ~\alpha \in \Phi_G^+ \setminus \Phi^+_L)
$$
which induce isomorphisms of $T$-modules
$$
\Gamma(Y, \cL_{\beta}) \cong \fg_{\alpha} \oplus \fg_{\theta(\alpha)}.
$$
In particular, $\varphi$ is surjective, and each invertible sheaf
$\cL_{\beta}$ is generated by its global sections. Moreover, the
corresponding morphism 
$$
F_\beta: Y \to \bbP(\fg_{\alpha} \oplus \fg_{\theta(\alpha)})^* 
\cong \bbP^1
$$
equals the morphism $f_{\alpha - \theta(\alpha)}$ (defined in
Prop.~\ref{prop:prod}).
\end{theorem}

\begin{proof}
(i) Since $Y$ is connected and fixed pointwise by $T_K$, each fiber
$\cN_{Y/X}(y)$, $y \in Y$, is a $T_K$-module, independent of the point
$y$. Considering the base point of $G/K$ and denoting by $\ft$
(resp.~$\fk$) the Lie algebra of $T$ (resp.~$K$), we obtain an
isomorphism of $T_K$-modules
$$
\cN_{Y/X}(y) \cong \fg/(\ft + \fk) \cong 
\bigoplus_{\alpha \in \Phi^+_G \setminus \Phi_L^+} \fg_{\alpha}
$$
which yields the decomposition (\ref{eqn:split}) of the normal sheaf
regarded as a $T_K$-linearized sheaf. Since the summands
$\cL_{\beta}$ are exactly the $T_K$-eigenspaces, they are stable
under $T$, and permuted by $N_K$ according to their weights.

(ii) Consider the restriction
$$
\varphi_0 :  \cO_{Y_0} \otimes \fg \to \cN_{Y_0/X_0}.
$$
By the isomorphism (\ref{eqn:ls}), the composite map
\begin{equation}\label{eqn:lt}
\cO_{Y_0} \otimes \fp_u \to \cO_{Y_0} \otimes \fg  \to \cN_{Y_0/X_0}
\end{equation}
is an isomorphism. Thus, $\varphi_0$ is surjective; by
$N_K$-equivariance, it follows that $\varphi$ is surjective as
well. Considering $T_K$-eigenspaces, this implies in turn the 
surjectivity of each $\varphi_{\beta}$.

Thus, the $T$-linearized sheaf $\cL_{\beta}$ is generated by a
$2$-dimensional $T$-module of global sections with weights $\alpha$
and $\theta(\alpha)$. This yields a $T$-equivariant morphism
$F_{\beta}: Y \to \bbP^1$. Its restriction to the open orbit 
$T/T_K$ is equivariant of weight $\alpha - \theta(\alpha)$ for the
action of $T$ by left multiplication; thus, we may identify this
restriction with the character $\alpha - \theta(\alpha)$. 
Now Proposition \ref{prop:prod} implies that 
$F_\beta = f_{\alpha - \theta(\alpha)}$. By (\ref{eqn:fxp}) and the
projection formula, it follows that the map 
$$
\fg_{\alpha} \oplus \fg_{\theta(\alpha)} = 
\Gamma(\bbP^1,\cO_{\bbP^1}(1)) \to 
\Gamma(Y,F_{\beta}^*\cO_{\bbP^1}(1)) = 
\Gamma(Y,\cL_{\beta}) 
$$
is an isomorphism.
\end{proof}

\subsection{The (logarithmic) tangent bundle}
\label{subsec:tb}

Recall that $\cT_X$ denotes the tangent sheaf of $X$, consisting
of all $k$-derivations of $\cO_X$. Let $\cS_X \subseteq \cT_X$
be the subsheaf consisting of derivations preserving the ideal sheaf
of the boundary $\partial X$. Since $\partial X$ is a divisor with
normal crossings, the sheaf $\cS_X$ is locally free; it is called
the \emph{logarithmic tangent sheaf} of the pair $(X,\partial X)$, 
also denoted by $\cT_X (- \log \partial X)$. 

Since $G$ acts on $X$ and preserves $X$, the map $op_X$ of 
(\ref{eqn:op}) factors through a map 
\begin{equation}\label{eqn:opr}
op_{X, \partial X} : \cO_X \otimes \fg \to \cS_X.
\end{equation}
In fact, $op_{X,\partial X}$ is surjective; this follows, e.g.,
from the local structure of $X$, see \cite[Prop.~2.3.1]{BB96} for
details. In other words, $\cS_X$ is the subsheaf of $\cT_X$
generated by the derivations arising from the $G$-action.

Clearly, the sheaf $\cT_X$ is $G$-linearized compatibly
with the subsheaf $\cS_X$. Moreover, the natural maps
$\cT_X \to \cN_{X_i/X} \cong \cO_X(X_i)\vert_{X_i}$
(where $X_1,\ldots,X_r$ denote the boundary divisors)
fit into an exact sequence of $G$-linearized sheaves
\begin{equation}\label{eqn:st}
0 \to \cS_X \to \cT_X \to 
\bigoplus_{i=1}^r \cN_{X_i/X} \to 0,
\end{equation}
see e.g. \cite[Prop.~2.3.2]{BB96}.

The pull-backs of $T_X$ and $\cS_X$ to $Y$ are described by the
following:

\begin{proposition}\label{prop:ex}
{\rm (i)} The exact sequence of $T N_K$-linearized sheaves
$$
0 \to \cT_Y \to \cT_X\vert_Y \to \cN_{Y/X} \to 0
$$
admits a unique splitting. 

{\rm (ii)} We also have a uniquely split exact sequence of
$T N_K$-linearized sheaves 
\begin{equation}\label{eqn:sn}
0 \to \cS_Y \to \cS_X \vert_Y \to
\cN_{Y/X} \to 0, 
\end{equation}
where $\cS_Y$ denotes the logarithmic tangent sheaf of the pair
$(Y,\partial Y)$. Moreover, the $T N_K$-linearized sheaf
$\cS_Y$ is isomorphic to $\cO_Y \otimes \fa$, where $T N_K$
acts on $\fa$ (the Lie algebra of $A$) via the natural action of its 
quotient $W_{G/K}$.  
\end{proposition}

\begin{proof}
(i) is checked by considering the $T_K$-eigenspaces as in the
proof of Theorem \ref{thm:split}. Specifically, the $T_K$-fixed
part of $\cT_X\vert_Y$ is exactly $\cT_Y$, while the sum
of all the $T_K$-eigenspaces with non-zero weights is mapped
isomorphically to $\cN_{Y/X}$.

(ii) First, note that the natural map 
\begin{equation}\label{eqn:ltb}
\cO_Y \otimes \fa \to \cS_Y
\end{equation}
is an isomorphism, since $Y$ is a nonsingular toric variety under
the torus $A/A \cap K$; see e.g. \cite[Prop.~3.1]{Od88}. 

Next, consider the map $\cT_X\vert_Y \to \cN_{Y/X}$ and
its restriction
$$
\pi : \cS_X\vert_Y \to \cN_{Y/X}.
$$
Clearly, the kernel of $\pi$ contains the image of the natural map
$$
i: \cO_Y \otimes \fa \to \cS_X\vert_Y.
$$
We claim that the resulting complex of $T N_K$-linearized sheaves
\begin{equation}\label{eqn:cplx}
\cO_Y\otimes \fa \to \cS_X\vert_Y \to \cN_{Y/X}
\end{equation}
is exact. By equivariance, it suffices to check this on $Y_0$. Then
the local structure (\ref{eqn:ls}) yields an exact sequence of
$P$-linearized sheaves
$$
0 \to \cO_{X_0} \otimes \fp_u \to
\cS_X\vert_{X_0}
\to \cO_{X_0}\otimes \fa \to 0,
$$
see \cite[Prop.~2.3.1]{BB96}. This yields, in turn, an isomorphism
$$
\cO_{Y_0}\otimes (\fp_u \oplus \fa) \cong
\cS_X\vert_{Y_0}
$$
which implies our claim by using the isomorphisms (\ref{eqn:lt}) and 
(\ref{eqn:ltb}).

In turn, this implies the exact sequence (\ref{eqn:sn}); its splitting
is shown by arguing as in (i).
\end{proof}

\begin{corollary}\label{cor:split}
We have isomorphisms of $T N_K$-linearized sheaves
\begin{equation}
\cT_X\vert_Y \cong \cT_Y \oplus 
\bigoplus_{\beta \in \Phi_K \setminus \Phi_L} \cL_{\beta},
\quad 
\cS_X\vert_Y \cong (\cO_Y\otimes \fa) \oplus 
\bigoplus_{\beta \in \Phi_K \setminus \Phi_L} \cL_{\beta},
\end{equation}
and an exact sequence of $T N_K$-linearized sheaves
\begin{equation}
0 \to \cO_Y \otimes \fa \to \cT_Y \to 
\bigoplus_{j=1}^m \cO_Y(Y_j)\vert_{Y_j} \to 0,
\end{equation}
where $Y_1, \ldots, Y_m$ denote the boundary divisors of the toric
variety $Y$.
\end{corollary}

\subsection{Equivariant Chern polynomials}
\label{subsec:ec}
By \cite{EG98}, any $G$-linearized locally free sheaf $\cE$ on $X$
yields \emph{equivariant Chern classes}
$$
c_i^G(\cE) \in A^i_G(X) \quad (i= 0, 1,\ldots, \rk(\cE))
$$
which we may encode by the \emph{equivariant Chern polynomial} 
$$
c_t^G(\cE) := \sum_{i=0}^{\rk(\cE)} c_i^G(\cE) \; t^i.
$$
The map 
$$
r : A^*_G(X) \to A^*_T(Y)^{W_K}
$$ 
of Theorem \ref{thm:eq} sends $c_t^G(\cE)$ to $c_t^T(\cE \vert_Y)$,
by functoriality of Chern classes. 

Together with the decompositions of the restrictions
$\cT_X\vert_Y$ and $\cS_X\vert_Y$ (Corollary \ref{cor:split}),
this yields product formulae for the equivariant Chern polynomials of
the $G$-linearized sheaves $\cT_X$ and $\cS_X$:

\begin{proposition}\label{ecp}
With the above notation, we have equalities in $A^*_T(Y)$:
\begin{equation}
r(c_t^G(\cS_X)) = \prod_{\beta \in \Phi_K \setminus \Phi_L} 
(1 + t \; c_1^T(\cL_{\beta})),
\end{equation}
\begin{equation}
r(c_t^G(\cT_X)) = \prod_{j=1}^m (1 + t \; [Y_j]_T) \times 
\prod_{\beta \in \Phi_K \setminus \Phi_L} (1 + t \; c_1^T(\cL_{\beta})),
\end{equation}
where $c_1^T(\cL_{\beta})\in \Pic_T(Y)$ denotes the equivariant Chern
class of the $T$-linearized invertible sheaf $\cL_{\beta}$, and 
$[Y_j]_T \in \Pic_T(Y)$ denotes the equivariant class of the
boundary divisor $Y_j$.
\end{proposition}

(Note that the above products are all $W_K$-invariant, but their
linear factors are not.) 

Likewise, we may express the image under $r$ of the 
\emph{equivariant Todd classes} 
$\td^G(\cT_X)$ and $\td^G(\cS_X)$:
$$
r(\td^G(\cS_X)) = \prod_{\beta \in \Phi_K \setminus \Phi_L} 
\frac{c_1^T(\cL_{\beta})}{1 - \exp(-c_1^T(\cL_{\beta}))},
$$
$$
r(\td^G(\cT_X)) = 
\prod_{j=1}^m 
\frac{[Y_j]_T}{1 - \exp(-[Y_j]_T)} \times
\prod_{\beta \in \Phi_K \setminus \Phi_L} 
\frac{c_1^T(\cL_{\beta})}{1 - \exp(-c_1^T(\cL_{\beta}))}.
$$

Finally, we determine the equivariant Chern classes 
$c_1^T(\cL_{\beta}) \in \Pic^T(Y)$
in terms of the boundary divisors of $Y$, indexed as in
Subsec.~\ref{subsec:tv}. Here $\beta \in \Phi_K \setminus \Phi_L$ and
hence $\beta = q(\alpha)$ for a unique 
$\alpha \in \Phi^+_G \setminus \Phi^+_L$.

\begin{proposition}\label{prop:rel}
With the preceding notation, we have
\begin{equation}\label{eqn:big}
c_1^T(\cL_{\beta}) = \alpha + 
\sum_{i,w} \langle \alpha - \theta(\alpha), w \omega_i^{\vee} \rangle 
\; [Y_{i,w}]_T,
\end{equation}
where $\omega_i^{\vee}$ denote the fundamental co-weights of the
restricted root system $\Phi_{G/K}$, and the sum runs over those pairs
$(i,w) \in E(\Phi_{G/K})$ such that 
$w^{-1}(\alpha - \theta(\alpha)) \in \Phi^+_{G/K}$.
\end{proposition}

\begin{proof}
Recall from Theorem \ref{thm:split} that the $T$-module of global
sections of $\cL_\beta$ is isomorphic to 
$\fg_{\alpha} \oplus \fg_{\theta(\alpha)}$. Let $s$ be the section of
$\cL_\beta$ associated with a generator of the line $\fg_\alpha$. Then
$c_1^T(\cL_\beta) = \alpha + \divi_T(s)$, since 
$s$ is a $T$-eigenvector of weight $\alpha$. Moreover, $\divi_T(s)$ is
the divisor of zeroes of the character $\alpha - \theta(\alpha)$. 
Together with (\ref{eqn:div}), this implies the equation (\ref{eqn:big}).
\end{proof}

\end{document}